\documentclass[11pt]{amsart}
\usepackage{amssymb}
\usepackage{amsmath}

\usepackage{color}
\usepackage[all]{xy}

\newtheorem{theorem}{Theorem}[section]
\newtheorem{lemma}[theorem]{Lemma}
\newtheorem{remark}[theorem]{Remark}
\newtheorem{definition}[theorem]{Definition}
\newtheorem{example}[theorem]{Example}
\newtheorem{proposition}[theorem]{Proposition}

\def\u{\mathfrak{A}}

\def\hbu{\mathcal H_{b\u}}

\def\PP{\mathcal P}
\def\C{\mathbb{C}}
\def\zC{\mathbb{C}}
\def\N{\mathbb{N}}

\def\T{\mathbb{T}}
\def\R{\mathbb{R}}

\def\p{\mathcal{P}}
\def\v{\overset\vee}
\def\l{\mathcal{L}}
\def\zH{\mathcal H}

\begin{document}

\title[Strongly mixing convolution operators on
Fr\'echet spaces]{Strongly mixing convolution operators on
Fr\'echet spaces of holomorphic functions}

\author{Santiago Muro, Dami\'an Pinasco, Mart\'{i}n Savransky}

\thanks{Partially supported by PIP 2010-2012 GI 11220090100624, PICT 2011-1456, UBACyT 20020100100746, ANPCyT PICT 11-0738 and CONICET}

\address{Santiago Muro
\hfill\break\indent Departamento de Matem\'{a}tica - Pab I,
Facultad de Cs. Exactas y Naturales, Universidad de Buenos Aires,
(1428), Ciudad Aut\'onoma de Buenos Aires, Argentina and CONICET} \email{{\tt smuro@dm.uba.ar}}

\address{Dami\'an Pinasco
\hfill\break\indent Departamento de Matem\'aticas y Estad\'{\i}stica,
Universidad Torcuato di Tella, Av. F. Alcorta 7350, (1428), Ciudad Aut\'onoma de Buenos Aires, ARGENTINA and
CONICET}
\email{{\tt dpinasco@utdt.edu}}

\address{Mart\'{i}n Savransky
\hfill\break\indent Departamento de Matem\'{a}tica - Pab I,
Facultad de Cs. Exactas y Naturales, Universidad de Buenos Aires,
(1428),Ciudad Aut\'onoma de Buenos Aires, Argentina and CONICET} \email{{\tt msavran@dm.uba.ar}}


\begin{abstract}
A theorem of Godefroy and Shapiro states that non-trivial convolution operators on the space of entire functions on $\mathbb{C}^n$ are hypercyclic. Moreover, it was shown by Bonilla and Grosse-Erdmann that they have frequently hypercyclic functions of exponential growth. On the other hand, in the infinite dimensional setting, the Godefroy-Shapiro theorem has been extended to several spaces of entire functions defined on Banach spaces. We prove that on all these spaces, non-trivial convolution operators are strongly mixing with respect to a gaussian probability measure of full support. For the proof we combine the results previously mentioned and we use techniques recently developed by Bayart and Matheron.  We also obtain the existence of frequently hypercyclic entire functions of exponential growth.

\end{abstract}


\maketitle


\section*{Introduction}

\bigskip

If $T$ is a continuous linear operator acting on some topological vector space $X$, the $T$-orbit of a vector $x\in X$ is the set $Orb(x,T):=\{x, Tx, T^2x,\dots\}$. The operator $T$ is said to be {\it hypercyclic } if there exist some vector $x\in X$, called {\it hypercyclic vector}, whose $T$-orbit is dense in $X$. Other forms of hypercyclicity where defined and studied in the literature. Specially, $T$ is {\it frequently hypercyclic} if there exist a vector $x\in X$, called {\it frequently hypercyclic vector}, whose $T$-orbit visits each non-empty open set along a set of integers having positive lower density.

Several criteria to determine if an operator is hypercyclic have been studied. It is known that a large supply of eigenvectors implies hypercyclicity. In particular,  if the eigenvectors associated to eigenvalues of modulus less than 1 and the eigenvectors associated to eigenvalues of modulus greater than 1 span dense subspaces, then the operator is hypercyclic. This result is due to Godefroy and Shapiro \cite{GodSha91}. They also prove there that non-trivial convolution operators, i.e. operators that commute with translations and which are not multiples of the identity, on the space of entire functions on $\C^n$ are hypercyclic.
This result has also been extended to some spaces of entire functions on infinite dimensional Banach spaces
(see \cite{AroBes99,BerBotFavJat13,CarDimMur07,Pet01,Pet06}).
The Godefroy - Shapiro theorem has been improved by Bonilla and Grosse-Erdmann. They showed that non-trivial convolution operators are even frequently hypercyclic, and have frequently hypercyclic entire functions satisfying some exponential growth condition (see \cite{BonGro06}).

Recent work developed by Bayart and Matheron \cite{BayMatSMALL} provides some other eigenvector criteria to determine whether a given continuous map $T : X \to X$ acting on a topological space $X$ admits an ergodic probability measure, or a strong mixing one. When the measure is strictly positive on any non void open set of $X$, ergodic properties on $T$ imply topological counterparts. In particular, if a continuous map $T : X \to X$ happens to be ergodic with respect to some Borel probability measure $\mu$ with full support, then almost every $x \in X$ (relative to $\mu$) has a dense $T$-orbit. Moreover, from Birkhoff's ergodic theorem, we can obtain frequent hypercyclicity.

In this article we study convolution operators on spaces of entire functions defined on Banach spaces. We show that under suitable conditions, non-trivial convolution operators are strongly mixing, and in particular, frequently hypercyclic. In the same spirit as Bonilla and Grosse-Erdmann, we also obtain the existence of frequently hypercyclic entire functions of exponential growth associated to these operators. We also prove the existence of frequently hypercyclic subspaces for a given non-trivial convolution operator, that is, the existence of closed infinite-dimensional subspaces in which every non-zero vector is a frequently hypercyclic function. Finally, we study particular cases of non-trivial convolution operators such as translations and partial differentiation operators. In this cases we obtain bounds of the exponential growth of the frequently hypercyclic entire functions.


\section{Holomorphic functions of $\u$-bounded type}\label{section hbu}

In this section we recall the basic properties of holomorphic functions on Banach spaces, the best general reference here is \cite{Din99}. We also introduce the spaces of entire functions $\hbu(E)$ and convolution operators therein.

Throughout this article $E$ is a complex Banach space.
A mapping $P: E \to \zC$ is a continuous $k$-homogeneous polynomial if there exists a (necessarily unique) continuous and symmetric $k$-linear form $L:E^k \to \zC$ such that $P(z)=L(z, \ldots, z)$ for all $z \in E$. For example, given $\gamma \in  E^{\prime}$, the function $P(z)=\gamma(z)^k$ is a $k-$homogeneous polynomial. The space of all continuous $k$-homogeneous polynomials from $E$ to $\mathbb{C}$, endowed with the norm $\Vert P \Vert_{\PP(^kE)}=\sup_{\Vert z \Vert_E=1} \vert P(z) \vert$ is a Banach space and it will be denoted by $\PP(^kE)$. The space $\PP(^0E)$ is just $\zC$. The space of {\it finite type} polynomials, denoted by $\PP_f(^kE)$, is the subspace of $\PP(^kE)$ spanned by $\{\gamma(\cdot)^k\}_{\gamma \in E^{\prime}}$.

The space of holomorphic functions from $E$ to $\mathbb C$ is denoted by $\mathcal H(E)$. If $f=\sum_{k\ge0} f_k$ is the Taylor series expansion of such a function, then it converges uniformly in some neighborhood around the point of expansion. The space of holomorphic functions whose Taylor series have infinite radius of uniform convergence is denoted $\mathcal H_b(E)$. Such functions are bounded on bounded sets, and are said to be of bounded type. The space $\mathcal H_b(E)$ is a Fr\'echet space when considered with the topology of uniform convergence on bounded sets of $E$.

Given $P\in\p(^kE)$, $a\in E$ and $0 \le j\le k$, let $P_{a^j}\in \p^{k-j}(E)$ be the polynomial defined by $$
P_{a^j}(x)=\v P(a^j,x^{k-j})=\v
P(\underbrace{a,...,a}_j,\underbrace{x,...,x}_{k-j}),$$ where $\v P$ is the unique symmetric $k$-linear form associated to $P$. We write $P_a$ instead of $P_{a^1}$.

Let us recall the definition of a polynomial ideal \cite{Flo01,Flo02}.
\begin{definition}\rm
A Banach ideal of scalar-valued continuous $k$-homogeneous polynomials, $k \ge 0$, is a pair
$(\mathfrak{A}_k,\|\cdot\|_{\mathfrak A_k})$ such that:
\begin{enumerate}
\item[(i)] For every Banach space $E$, $\mathfrak{A}_k(E)=\mathfrak A_k\cap
\mathcal
P(^kE)$ is a linear subspace of $\p(^kE)$ and $\|\cdot\|_{\u_k(E)}$
is a norm on it. Moreover, $(\u_k(E), \|\cdot\|_{\u_k(E)})$ is a
Banach space.

\item[(ii)] If $T\in \l (E_1,E)$ and $P \in \u_k(E)$, then $P\circ T\in
\u_k(E_1)$ with $$ \|
P\circ T\|_{\u_k(E_1)}\le \|P\|_{\u_k(E)} \| T\|^k.$$

\item[(iii)] $z\mapsto z^k$ belongs to $\u_k(\mathbb C)$
and has norm 1.
\end{enumerate}
\end{definition}

The concept of holomorphy type was introduced by Nachbin \cite{Nac69}. We will use it in the following slightly modified version (see \cite{Mur12}).
\begin{definition}\label{defihtype}\rm
Consider the sequence $\u=\{\u_k\}_{k=0}^\infty$, where for each
$k$, $\u_k$ is a Banach ideal of  $k$-homogeneous
polynomials. We say that $\{\u_k\}_k$ is a holomorphy type if there exists constants $c,\, c_{k,l}$ such that $c_{k,l}\le c^k$ for every $0 \le l \le k$ and such that for every Banach space $E$, $P\in \u_{k}(E)$ and $a\in E$,
\begin{equation}\label{htype}
P_{a^l}\textrm{ belongs to } \u_{k-l}(E)\textrm{ and } \|P_{a^l}\|_{\u_{k-l}(E)} \le c_{k,l}
\|P\|_{\u_{k}(E)} \|a\|^l.
\end{equation}
\end{definition}

There is a natural way to associate to a holomorphy type $\u$ a class of entire functions of bounded type on a Banach
space $E$, as the set of entire functions with infinite $\u$-radius of convergence at zero, and hence at every point (see \cite{CarDimMur07,FavJat09}).
\begin{definition}\rm
Let $\u=\{\u_k\}_k$ be a holomorphy type and $E$ be a Banach space. The space of entire functions of
$\u$-bounded type on $E$ is the set 
$$
\hbu(E):=\left\{ f\in \mathcal H(E): \;  d^kf(0)\in \u_k(E) \textrm{ for every }k \textrm{ and }  \lim_{k\rightarrow \infty}
\Big\|\frac{d^kf(0)}{k!}\Big\|_{\u_k}^{1/k}=0\right\}.
$$
\end{definition}
We consider in $\hbu(E)$ the family of seminorms $\{p_s\}_{s>0}$, given by
$$
p_s(f)=\sum_{k=0}^{\infty} s^k
\left\|\frac{d^kf(0)}{k!}\right\|_{\u_k},
$$
for all $f\in \hbu(E)$. It is easy to check that $\left(\hbu(E),\{p_s\}_{s>0}\right)$ is a Fr\'echet space.

The following example collects some of the spaces of entire functions of bounded type that may be constructed in this way. See the references given in each case for the definition and details.
\begin{example}\label{ejemplos htype}\rm
\begin{enumerate}
\item[(i)]If we let $\u_k=\p^k$, the ideal of all $k$-homogeneous continuous polynomials, then the topology induced on $\hbu(E)$ by $\{p_s\}_{s>0}$ is equivalent to the usual topology of uniform convergence on bounded sets. Therefore $\hbu(E)=\mathcal H_b(E)$.
\item[(ii)] If $\u$ is the sequence of ideals of nuclear polynomials then $\hbu(E)$
is the space of holomorphic functions of nuclear bounded type
$H_{Nb}(E)$ defined by Gupta and Nachbin (see \cite{Gup70}).
\item[(iii)] If $E$ is a Hilbert space and $\u$ is the sequence of ideals of Hilbert-Schmidt polynomials, then $\hbu(E)$ is the space $H_{hs}(E)$ of entire functions of Hilbert-Schmidt type (see \cite{Dwy71,Pet01}).
\item[(iv)] If $\u$ is the sequence of ideals of approximable polynomials, then $\hbu(E)$ is the space $H_{bc}(E)$ of entire functions of compact bounded type (see for example \cite{Aro79,AroBes99}).
\item[(v)] If $\u$ is the sequence of ideals of weakly continuous on bounded sets
polynomials, then $\hbu(E)$ is the space $H_{wu}(E)$ of weakly uniformly continuous holomorphic functions on bounded sets defined by Aron in \cite{Aro79}.
\item[(vi)] If $\u$ is the sequence of ideals of extendible polynomials,
then $\hbu(E)$ is the space of extendible functions of bounded type defined in \cite{Car01}.
\item[(vii)] If $\u$ is the sequence of ideals of integral polynomials, then $\hbu(B_E)$ is the space of
integral holomorphic functions of bounded type $H_{bI}(B_E)$ defined
in \cite{DimGalMaeZal04}.
\end{enumerate}
\end{example}

Given $\u=\{\u_k\}_k$ a holomorphy type, the Borel transform is the  operator $\beta:\hbu(E)'\to \zH(E')$
which assigns to each element $\varphi\in\hbu(E)'$ the holomorphic function $\beta(\varphi)\in \zH(E'),$
given by $\beta(\varphi)(\gamma)=\varphi(e^\gamma)$. The following proposition is well-known (see for
example \cite[p.264]{Din71(holomorphy-types)} or \cite[p.915]{FavJat09}).

\begin{proposition}\label{beta_inyectiva}
Let $\u=\{\u_k\}_k$ be a holomorphy type and $E$ be a Banach space such that finite type polynomials are dense
in $\u_k(E)$ for every $k$. Then the Borel transform is an injective linear transformation.
\end{proposition}


Finite type polynomials are dense in $\u_k(E)$ in many cases. For example, finite type polynomials are dense
in the spaces of nuclear, Hilbert-Schmidt and approximable polynomials. They are also dense in $\mathcal
P(^kE)$ if $E$ is $c_0$ or the Tsirelson space and in the spaces of integral and extendible polynomials if $E$
is Asplund \cite{CarGal11(radon-nikodym)}. On the other hand, separability is a necessary condition to deal with hypercyclicity issues on $\hbu(E)$ and, up to our knowledge, on every example of separable space of polynomials, finite type polynomials are dense.

We also note that a holomorphy type such that finite type polynomials are dense is essentially what is
called an $\alpha$-$\beta$-holomorphy type in \cite{Din71(holomorphy-types)} and a $\pi_1$-holomorphy type in
\cite{FavJat09,BerBotFavJat13}.

We denote by $\tau_x(f):=f(x+\cdot)$ the translation operator by $x$, which is a continuous linear
operator on $\hbu(E)$ (see \cite{CarDimMur07,FavJat09}).
The following is the usual definition of convolution operator.

\begin{definition}\rm
Let $\u=\{\u_k\}_k$ be a holomorphy type and $E$ be a Banach space. A linear continuous operator $T$ defined on $\hbu(E)$
is a  {\it convolution operator}, if for every $x\in E$ we have $T\circ\tau_x=\tau_x\circ T$. We say that $T$ is trivial if it is a multiple
of the identity.
\end{definition}

The following proposition provides a description of convolution operators on $\hbu(E)$. Its proof follows as \cite[Proposition 4.7]{CarDimMur07}.
\begin{proposition}\label{convolution}
Let $\u=\{\u_k\}_k$ be a holomorphy type and $E$ be a Banach space. Then for each convolution operator
$T:\hbu(E)\to\hbu(E)$
there exists a linear functional $\varphi\in\hbu(E)'$ such that
$$
T(f) =\varphi\ast f,
$$
for every $f \in\hbu(E)$, where $\varphi\ast f (x):=\varphi(\tau_x f)=\varphi(f(x+\cdot))$.
\end{proposition}
\begin{proof}
Let $\varphi =\delta_0\circ T$, i.e.
$\varphi(f) = T(f)(0)$ for $f \in\hbu(E)$. Then $\varphi\in\hbu(E)'$ and
$$
T(f)(x) = \left[\tau_x T(f) \right](0) = T(\tau_x f)(0) = \varphi(\tau_x f) = \varphi\ast f(x),
$$
for every $f\in\hbu(E)$ and $x \in E$.
\end{proof}

\section{Strongly mixing convolution operators}

In this section we prove our first main theorem, which states that under some fairly general conditions on the space $E$ and the holomorphy type $\u$, non-trivial convolution operators on $\hbu(E)$ are strongly mixing in the gaussian sense.
First we recall the following definitions.
\begin{definition}\rm
A Borel probability measure on $X$ is gaussian if and only if it is the
distribution of an almost surely convergent random series of the form $\xi=\sum_{n=0}^\infty{g_nx_n}$, where $(x_n) \subset X$ and $(g_n)$ is a sequence of independent, standard complex gaussian variables.
\end{definition}

Recall that for an operator $T:X\to X$ we say that $\mu$ is a $T$-invariant ergodic measure if $\mu(A)=\mu(T^{-1}A)$ for all measurable sets $A\subset X$ and if given $A$, $B$ measurable sets of positive measure then one can find an integer $n\ge 0$ such that $T^n(A)\cap B \neq \emptyset$. When the measure $\mu$ is strictly positive on all non void open sets, ergodicity implies that $T$ is topologically transitive, hence hypercyclic. Additionally, Birkhoff's ergodic theorem implies that $T$ is frequently hypercyclic.

We are specially interested in a condition stronger that ergodicity, namely {\it strongly mixing} with respect to some gaussian probability measure.

\begin{definition}\rm
We say that an operator $T\in L(X)$ is strongly mixing in the gaussian sense if there exists some gaussian $T$-invariant probability measure $\mu$ on $X$ with full support such that for any measurable sets $A,\,B\subset X$ we have
$$
\lim_{n\to\infty}\mu(A\cap T^{-n}(B))=\mu(A)\mu(B).
$$
\end{definition}

We will use the following theorem due to Bayart and Matheron, which is a corollary of \cite[Theorem 1.1]{BayMatSMALL}.

\begin{theorem}[Bayart-Matheron]\label{Bayart_Matheron}
Let $X$ be a complex separable Fr\'echet space, and let $T \in L(X)$. Assume that for any $D \subset \T$ such that $\T\setminus D$ is dense in $\T$, the linear span of
$\bigcup_{\lambda\in\T\setminus D} \ker(T-\lambda)$ is dense in $X$. Then $T$ is strongly mixing in the gaussian sense. In particular, $T$ is a frequently hypercyclic operator.
\end{theorem}

The next lemma is the key to prove that convolution operators are strongly mixing and it we will be used throughout the article.

\begin{lemma}\label{eigen_densos}
Let $E$ be a Banach space with separable dual and let $\u$ be a holomorphy type such that finite
type polynomials are dense in $\u_k(E)$ for every $k$. Let $\phi\in \zH(E')$ not constant
and $B\subset \C$.
Suppose that there exist $\gamma_0\in E'$ such that $\phi(\gamma_0)$ is an accumulation point of $B$. Then
$span\{e^\gamma: \, \phi(\gamma)\in B\}$ is dense in $\hbu(E)$.
\end{lemma}

\begin{proof}
Let $\Phi\in \hbu(E)'$ be a functional vanishing on $\{e^\gamma: \, \phi(\gamma)\in B\}$. Note that this means that $\beta(\Phi)$ vanishes on $\phi^{-1}(B)$. By Proposition \ref{beta_inyectiva}, it suffices to show that $\beta(\Phi)=0$.

Fix $\gamma_0\in E'$ such that $\phi(\gamma_0)$ is an accumulation point of $B$. We claim that there exist a
sequence of complex lines $L_k$, $k\in\N$,
through $\gamma_0$ such that $\phi$ is not constant on each $L_k$ and $\bigcup_k L_k$ is dense in $E'$.
Indeed, let $\{U_k\}_{k\in\mathbb{N}}$, be open sets that form a basis of the topology of $E'$. Since $\phi$
is not constant, there exists, for each $k$, a complex line $L_k$ through $\gamma_0$
that meets $U_k$ and on which $\phi$ is not constant.

Now let $k\in \N$. Since $\phi$ is not constant on $L_k$, $\phi|_{L_k}$ is an open mapping,
and hence $\gamma_0$ is an accumulation point of $\phi^{-1}(B)\cap L_k$. But, $\beta(\Phi)$ vanishes
on $\phi^{-1}(B)$. Thus, $\beta(\Phi)$ also vanishes on $L_k$. Since $\bigcup_k L_k$ is dense in $E'$,
$\beta(\Phi)=0$.
\end{proof}

We are now able to prove that convolution operators on $\hbu(E)$ are strongly mixing in the gaussian sense.

\begin{theorem}
Let $\u=\{\u_k\}_k$ be a holomorphy type and $E$ a Banach space with separable dual such that the finite type polynomials are dense in $\u_k(E)$ for every $k$. If $T:\hbu(E)\to\hbu(E)$ is a non-trivial convolution operator, then $T$ is strongly mixing in the gaussian sense.
\end{theorem}

\begin{proof}
Let $\varphi\in\hbu(E)'$ be the linear functional defined in the proof of Proposition \ref{convolution}.
Since $T$ is not a multiple of the identity it follows that $\varphi$ is not a multiple of $\delta_0$.
Also, the fact that $\varphi$ is not a multiple of $\delta_0$ implies that $\beta(\varphi)$ is not a constant function.
Indeed, if $\beta(\varphi)$ were constant then $\lambda := \varphi(1) = \beta(\varphi)(0) = \beta(\varphi)(\gamma) = \varphi(e^\gamma)$ for all $\gamma\in E'$. But, on the other hand, $\lambda = \lambda\delta_0(e^\gamma)$ for all $\gamma\in E'$ and
we would have that $\varphi = \lambda\delta_0$.

It is rather easy to find eigenvalues and eigenvectors for $T$. Given $\gamma\in E'$,
$$
T(e^\gamma) = \varphi\ast e^\gamma = \left[x\mapsto \varphi(\tau_x e^\gamma) \right]
= \varphi(e^\gamma)\left[x\mapsto e^{\gamma(x)}\right] = \beta(\varphi)(\gamma) e^\gamma.
$$
By Theorem \ref{Bayart_Matheron}, it suffices to prove that the set of unimodular eigenvectors $\{e^\gamma\in \hbu(E)\,: |\beta(\varphi)(\gamma)|=1\}$ is big enough. Let us first prove that it is not empty. Define
$$
V=\{\gamma\in E' :\; |\beta(\varphi)(\gamma)|<1\} \; \textrm{  and  } \; W=\{\gamma\in E': |\beta(\varphi)(\gamma)|>1\}.
$$
Let us check that $V,W \subset E'$ are non void open sets. Indeed, if $V = \emptyset$, or $W = \emptyset$, then $\frac{1}{\beta(\varphi)}$, or $\beta(\varphi)$, would be a nonconstant bounded entire function. Since $\beta(\varphi)(E')$ is arcwise connected, we can deduce the existence of $\gamma_0\in E'$ such that $|\beta(\varphi)(\gamma_0)|=1$.

Let $D\subset\T$ such that $\T \setminus D$ is dense in $\T$. Then, $\beta(\varphi)(\gamma_0)$ is an accumulation point of $\T\setminus D$ and by Lemma \ref{eigen_densos} we get that the linear span of $\bigcup_{\lambda\in\T \setminus D} \ker(T-\lambda)$ is dense in $\hbu(E)$. By Theorem \ref{Bayart_Matheron}, it follows that $T$ is strongly mixing in the gaussian sense, as we wanted to prove.
\end{proof}

\section{Exponential growth conditions for frequently hypercyclic entire functions}

In this section we show that for every convolution operator there exists a frequently hypercyclic entire function satisfying a certain exponential growth condition.
First, we define and study a family of Fr\'echet subspaces of $\hbu(E)$ consisting of functions of exponential type; and then we show that every convolution operator on $\hbu(E)$ defines a frequently hypercyclic operator on these spaces.

\begin{definition}\rm
A function $f\in\hbu(E)$ is said to be of $M$-exponential type if there exist some constant $C>0$ such that $|f(x)|\leq C e^{M\|x\|}$, for all $x\in E$. We say that $f$ is of exponential type if it is of $M$-exponential type for some $M>0$.
\end{definition}

Now we define the subspaces of $\hbu(E)$ consisting of functions of exponential type.

\begin{definition}\rm
For $p>0$, let us define the space
$$
Exp_{\u}^p(E)=\left\{f\in \hbu(E) : \limsup_{k\to\infty} \|d^kf(0)\|_{\u_k}^{1/k}\leq p\right\},$$
endowed with the family of seminorms defined by
$$
q_r(f)=\sum_{k=0}^{\infty} r^k\|d^kf(0)\|_{\u_k} \quad \mbox{for } 0<r<1/p.
$$
\end{definition}

Below we collect some basic properties of the spaces $Exp_{\u}^p(E)$. Their proof is standard.

\begin{proposition}
Let $p$ be a positive number and $\u=\{\u_k\}_k$ a holomorphy type.
\begin{itemize}
\item[(a)] A function $f\in H(E)$ belongs to $Exp_{\u}^p(E)$ if and only if $d^kf(0)\in \u_k$ for all $k\in
\N$ and $q_r(f)<\infty,$ for all $0<r<1/p$.

\item[(b)] The space $(Exp_{\u}^p(E),\{q_r\}_{r<1/p})$ is a Fr\'echet space that is continuously and
densely embedded in $\hbu(E)$.

\item[(c)] If $E'$ is separable and finite type polynomials are dense in $\u_k(E)$ for every $k$, then $Exp_{\u}^p(E)$ is separable.

\item[(d)] Every function $f\in Exp_{\u}^p(E)$ satisfies the following growth condition: for each $\varepsilon>0$, there exists $C_\varepsilon>0$ such that
$$
|f(x)|\leq C_\varepsilon e^{(p+\varepsilon)\|x\|}, \;\;\; x\in E,
$$
that is, $f$ is of exponential type $p$.
\end{itemize}
\end{proposition}

In order to prove frequent hypercyclicity of convolution operators on $Exp_{\u}^p(E)$, we need to introduce
some structure on the sequence of polynomial ideals.

\begin{definition}\rm
Let $\u=\{\u_k\}_k$ a holomorphy type and let $E$ be
a Banach space. We say that $\u$ is weakly differentiable at $E$ if there
exist constants $c_{k,l}>0$ such that, for  $0 \le l \le k$,  $P\in\u_k(E)$ and $\varphi\in\u_{k-l}(E)'$,  the mapping
$x\mapsto\varphi(P_{x^l})$ belongs to $\u_l(E)$ and
$$\Big\|x\mapsto\varphi\big(P_{x^l}\big)\Big\|_{\u_l(E)} \le c_{k,l}\|\varphi\|_{\u_{k-l}(E)'}\|P\|_{\u_k(E)}.$$
\end{definition}

\begin{remark}\label{stirling}\rm
In the following, we will assume that
\begin{equation}\label{constantes}
 c_{k,l} \le\frac{(k+l)^{k+l}}{(k+l)!}\frac{k!}{k^k}\frac{l!}{l^l} \mbox{ for every } k,l.
\end{equation}
Stirling's Formula states that $e^{-1}n^{n+1/2}\le e^{n-1}n!\le
n^{n+1/2}$ for every $n\ge1$, so given $\varepsilon>0$, there exists a positive constant $c_\varepsilon$, such that
\begin{equation*}\label{eq stirling}
c_{k,l}\le e^2\Big(\frac{kl}{k+l}\Big)^{1/2} \le c_\varepsilon(1+\varepsilon)^k,
\end{equation*}
for every $0 \le l \le k$.
\end{remark}
\begin{remark}\label{w-dif dual a mult}\rm
Weak differentiability is a condition that is stronger than being a holomorphy type and was defined in
\cite{CarDimMurCOLL}. All  the spaces of entire functions appearing in Example \ref{ejemplos htype} are
constructed with weakly differentiable holomorphy types satisfying (\ref{constantes}), see
\cite{CarDimMurCOLL,Mur12}. The concept of weak differentiability is closely related to that
of $\alpha$-$\beta$-$\gamma$-holomorphy types in \cite{Din71(holomorphy-types)} and that of
$\pi_1$-$\pi_2$-holomorphy types in \cite{FavJat09,BerBotFavJat13}.
\end{remark}

\begin{proposition}\label{restriccion_exp}
Let $p$ be a positive number, $\u=\{\u_k\}_k$ a holomorphy type and let $E$ be a Banach. Suppose that $\u$ is weakly differentiable with constants $c_{k,l}$ satisfying (\ref{constantes}). Then every convolution operator on $\hbu(E)$, restricts to  a convolution operator on $Exp_{\u}^p(E)$.
\end{proposition}

\begin{proof}
Let $T:\hbu(E)\to\hbu(E)$ be a convolution operator and $\varphi\in \hbu(E)'$ such that $Tf=\varphi \ast f$. Suppose that $f=\sum_{k\in\N_0} P_k$ is in $Exp_{\u}^p(E)$. We need to prove that for $r<1/p$
$$
q_r(\varphi\ast f)=\sum_{l=0}^\infty r^l\|d^l(\varphi\ast f)(0)\|_{\u_l}<\infty.
$$

Note that
$$\varphi\ast f(x) = \varphi(\tau_x f) = \varphi\left( \sum_{k=0}^\infty \sum_{l=0}^k  \binom{k}{l}
(P_k)_{x^l} \right) = \sum_{l=0}^\infty \sum_{k=l}^\infty \binom{k}{l} \varphi \left((P_k)_{x^l}\right).
$$
This implies that
$$
d^l(\varphi\ast f)(0)(x)=l!\sum_{k=l}^{\infty}\binom{k}{l}\varphi((P_k)_{x^l}).
$$

Since $\varphi$ is a continuous linear functional, there are positive constants $c$ and $M$ such that $\|\varphi\|_{\u_{k-l}'}\leq c M^{k-l}.$
Thus, given $\varepsilon>0$ such that $r(1+\varepsilon)<1/p$, by the above remark,

\begin{align*}
q_r(\varphi\ast f) = \sum_{l=0}^\infty r^l\|d^l(\varphi\ast f)(0)\|_{\u_l}&\leq \sum_{l=0}^\infty r^l l! \sum_{k=l}^{\infty}\binom{k}{l}\|x\mapsto
\varphi((P_k)_{x^l})\|_{\u_k}\\
&\leq \sum_{l=0}^\infty r^l l! \sum_{k=l}^{\infty}\binom{k}{l} c_{k,\,l}\|\varphi\|_{\u_{k-l}'}\|P_k\|_{\u_k} \\
&\leq c\sum_{k=0}^\infty \frac{\|d^kf(0)\|_{\u_k}}{k!}\sum_{l=0}^k\binom{k}{l} c_{k,\,l} \, r^l \, l! \, M^{k-l} \\
&\leq c\sum_{k=0}^\infty \|d^kf(0)\|_{\u_k} r^k c_\varepsilon(1+\varepsilon)^k\sum_{l=0}^k
\frac{\left(\frac{M}{r}\right)^{k-l}}{(k-l)!}\\
&\leq c\, c_\varepsilon\, e^{(M/r)}\sum_{k=0}^\infty \|d^kf(0)\|_{\u_k}(r(1+\varepsilon))^k\\
&=c\, c_\varepsilon \, e^{(M/r)}\, q_{r(1+\varepsilon)}(f)<\infty.
\end{align*}
\end{proof}

\begin{remark}\rm
For $\gamma\in E'$, we have $d^k(e^\gamma)(0)=\gamma^k$, and then, since $\|\gamma^k\|_{\u_k}= \| \gamma \|^k$,
$$
\limsup_{k\to\infty}\|d^ke^\gamma(0)\|_{\u_k}^{1/k}=\|\gamma\|_{E'}.
$$
This implies that $e^\gamma\in Exp_\u^p(E)$ if and only if $\|\gamma\| \le p$.
Thus, for $\varphi\in Exp_\u^p(E)'$, we can define the Borel transform $\beta(\varphi)(\gamma)=\varphi(e^\gamma)$, for all $\gamma\in E'$ with $\|\gamma\| \le p$. Moreover, the function $\beta(\varphi)$ is holomorphic on the set $p\, B_{E\,'}$.
\end{remark}

The next proposition is the analogue   of Proposition \ref{beta_inyectiva} for the Borel transform restricted
to $Exp_{\u}^p(E)$.

\begin{proposition}\label{beta_inyectiva_exp}
Let $\u=\{\u_k\}_k$ be a holomorphy type and $E$ a  Banach space such that finite type polynomials are dense
in $\u_k(E)$ for every $k$. Then the Borel transform $\beta:Exp_{\u}^p(E)'\to \mathcal H(p B_{E'})$ is an injective
linear transformation.
\end{proposition}

Now, we can restate Lemma \ref{eigen_densos}, for the space $Exp_{\u}^p(E)$. Its proof is similar.

\begin{lemma}\label{eigen_densos_exp}
Let $p$ be a positive number, let $E$ be a Banach space with separable dual and let $\u$ be a holomorphy type such that finite type polynomials are dense in $\u_k(E)$ for every $k$. Let $\phi\in \mathcal H(pB_{E'})$ not constant
and $B\subset \C$. Suppose that there exist $\gamma_0\in p\,B_{E'}$ such that $\phi(\gamma_0)$ is an accumulation point of $B$. Then
$span\{e^\gamma: \, \|\gamma\|<p,\, \phi(\gamma)\in B\}$ is dense in $Exp_{\u}^p(E)$.
\end{lemma}

Now we are able to prove that for non-trivial convolution operators on $\hbu(E)$ there exist frequently hypercyclic entire function satisfying certain exponential growth conditions.

Given a non-trivial convolution operator $T$ defined on $\hbu(E)$, let us define
\[
\alpha_T=\inf\{\|\gamma\|,\,\gamma\in E'\text{ such that } |T(e^\gamma)(0)|=1\}.
\]

\begin{theorem}\label{ultimo}
Let $\u=\{\u_k\}_k$ be a holomorphy type and let $E$ be a Banach space with separable dual such that
finite type polynomials are dense in $\u_k(E)$ for every $k$. Suppose that $\u$ is weakly differentiable with
constants $c_{k,l}$ satisfying (\ref{constantes}). Let $T:\hbu(E)\to\hbu(E)$ be a non-trivial convolution
operator. Then, for any $\varepsilon>0$, $T$ admits a frequent hypercyclic function $f \in Exp_{\u}^{\alpha_T+\varepsilon}(E)$.
\end{theorem}

\begin{proof}
Fix $\gamma_0\in E'$ such that $\alpha_T\leq \|\gamma_0\|<\alpha_T+\varepsilon$ and
$|T(e^{\gamma_0})(0)|=1$. Consider $p=\alpha_T+\varepsilon$. It is enough to prove that $T$ is
frequently hypercyclic on $Exp_\u^p(E)$.

Proposition \ref{restriccion_exp} allows us to restrict the operator $T$ to the space $Exp_{\u}^p(E)$. Since
$e^\gamma$ is an eigenvector of $T$ with eigenvalue $T(e^\gamma)(0)$,  it is
enough to show, by Theorem \ref{Bayart_Matheron}, that for every Borel set $D\subset\T$, such that $\T \setminus D$ is dense in $\T$, the linear span of $\{e^\gamma:\,\|\gamma\|<p,\, T(e^\gamma)(0)\in \T\setminus D\}$ is dense in
$Exp_{\u}^p(E)$.

We see as in Proposition \ref{convolution} that there exists $\varphi\in Exp_{\u}^p(E)'$ such that $Tf=\varphi\ast f$ for every $f\in Exp_{\u}^p(E)$. Then
$\beta(\varphi)\in \zH(p\,B_{E'})$ is not constant. Since $\T\setminus
D$ is dense in $\T$, $T(e^{\gamma_0})(0)=\beta(\varphi)(\gamma_0)$ is an accumulation point of $\T\setminus
D$. Thus, an application of Lemma \ref{eigen_densos_exp} proves that the linear span of $\{e^\gamma:\,\|\gamma\|<p,\,
\beta(\varphi)(\gamma)\in \T\setminus D\}$ is dense in
$Exp_{\u}^p(E)$.
\end{proof}

\section{Frequently hypercyclic subspaces and Examples}

Finally, we study the existence of frequently hypercyclic subspaces for a given non-trivial convolution
operator, that is, the existence of closed infinite-dimensional subspaces in which every non-zero vector is
frequently hypercyclic. We prove that there exists a frequently hypercyclic subspace for each
non-trivial convolution operator on $\hbu(E)$, if the dimension of $E$ is bigger than 1.

Lastly, we study exponential growth conditions for special cases of convolution operators such as translation and partially differentiation ones.

\subsection{Frequently hypercyclic subspaces} Given a frequently hypercyclic operator $T$ on a Fr\'echet space
$X$ with frequently hypercyclic vector $x\in X$, we can consider the linear subspace $\mathbb{K}[T]x$, whose
elements are the evaluations at $x$ of every polynomial on $T$. It turns out that $\mathbb{K}[T]x\setminus
\{0\}$ is contained on $FHC(T)$, the set of all frequently hypercyclic vectors of $T$, but in general
$\mathbb{K}[T]x$ is not closed in $X$. Then, it is natural to ask if there exists a {\it closed} subspace
$M\subset X$ such that $M\setminus \{0\}\subset FHC(T)$. Bonilla and Grosse-Erdmann, in \cite{BonGro12},
gave sufficient conditions for this situation to hold. First we state the Frequent Hypercyclicity Criterion.

\begin{theorem}[Frequent Hypercyclicity Criterion]\label{FHC}
Let $T$ be an operator on a separable F-space $X$. Suppose that there exists a dense subset $X_0$ of $X$ and a map $S : X_0 \to X_0$
such that, for all $x \in X_0$,
\begin{enumerate}
\item $\sum_{n=1}^\infty T^n x$ converges unconditionally,
\item $\sum_{n=1}^\infty S^n x$ converges unconditionally,
\item $T Sx = x$.
\end{enumerate}
Then $T$ is frequently hypercyclic.
\end{theorem}

The Bonilla and Grosse-Erdmann theorem for the existence of a frequently hypercyclic subspace states that if
an operator $T$ satisfies the Frequent Hypercyclicity Criterion and admits an infinite number of linearly
independent eigenvectors, associated to an eigenvalue of modulus less than one then, there exists a frequently
hypercyclic subspace for $T$. Since we cannot assure that non-trivial convolution operators satisfy the
Frequent Hypercyclicity Criterion, Theorem \ref{FHC}, we need the following modified version which may be
found in \cite[Remark 9.10]{GroPer11}.
\begin{proposition}\label{FHC modified}
Let $T$ be an operator on a separable F-space $X$. Suppose that there exists a dense subset $X_0$ of $X$
and for any $x\in X_0$ there is a sequence $(u_n(x))_{n\ge 0}\subset X$
such that,
\begin{enumerate}
\item $\sum_{n=1}^\infty T^n x$ converges unconditionally,
\item $\sum_{n=1}^\infty u_n(x)$ converges unconditionally,
\item $T^ju_n(x) =u_{n-j}(x)$, for $j\le n$.
\end{enumerate}
Then $T$ is frequently hypercyclic.
\end{proposition}

Now, we can state the modified version of the Bonilla and Grosse-Erdmann theorem which will be used for the
proof
of Theorem \ref{FHS convolution}.
\begin{theorem}\label{subspacio}
Let $X$ be a separable F-space with a continuous norm and $T$ an operator on $X$ that satisfies the
hypotheses of Proposition \ref{FHC modified}. If $\dim  Ker(T - \lambda) =
\infty$ for some scalar $\lambda$ with $|\lambda|<1$ then $T$ has a frequently hypercyclic subspace.
\end{theorem}
The proof of the previous theorem follows the same lines as the proof of \cite[Theorem 3]{BonGro12}, but
replacing
$S^ny_j$ by $u_n(y_j)$, for each $y_j\in X_0$, in their key Lemma 1.
Next, we prove the existence of frequent hypercyclic subspaces for every non-trivial convolution operator, if
$dim(E)>1$. The corresponding problem for $dim(E)=1$ is open, up to our knowledge.

\begin{theorem}\label{FHS convolution}
Let $\u=\{\u_k\}_k$ be a holomorphy type and $E$ a Banach space with $dim(E)>1$ and separable dual such that
the finite type polynomials are dense in $\u_k(E)$ for every $k$. If $T:\hbu(E)\to\hbu(E)$ is a
non-trivial convolution operator, then $T$ has a frequently hypercyclic subspace.
\end{theorem}
\begin{proof}
Let us see that both hypotheses of Theorem \ref{subspacio} are fulfilled by every non-trivial convolution operator on $\hbu(E)$.
Recall that if $T:\hbu(E) \to \hbu(E)$ is a non-trivial convolution operator then
$\beta(\varphi)(\gamma)=T(e^\gamma)(0)$ is holomorphic as a function of $\gamma\in E'$, and that
$T(e^\gamma)= [T(e^\gamma)(0)]e^\gamma$.
We have that $\{e^\gamma\, :\, \gamma\in E'\}$ is a linearly independent set in $\hbu(E)$, see \cite[Lemma 2.3]{AroBes99}. We will prove that there exists some
scalar $\lambda$ with $|\lambda|<1$ such that $\dim  Ker(T - \lambda) =\infty$. We follow the ideas of the proof of \cite[Theorem 5]{Pet06}.
If the set of zeros of $\beta(\varphi)$, denoted by $Z(\beta(\varphi))=\{\gamma\in E' :
\beta(\varphi)(\gamma)=0\}$, is infinite then we take $\lambda=0$, because
$Ker(T)\supset \{e^\gamma: \gamma\in Z(\beta(\varphi))\}$.
If $Z(\beta(\varphi))$ is not infinite, then it is empty since $\dim(E)>1$. Now, fix $\gamma\in E'$ and consider $f_\gamma(w)=\beta(\varphi)(w\gamma)$ for $w\in\zC$.
From the continuity of $T$ and of $\delta_0$, we get that there exist positive constants $M$ and $s$ such that
\begin{align*}
 |f_\gamma(w)|&=|T(e^{w\gamma})(0)| \leq M p_s(e^{w\gamma}) = M \sum_{k\geq 0}
\frac{s^k}{k!}\|d^k(e^{w\gamma})(0)\|_{\u_k}\\
 & = M \sum_{k\geq 0} \frac{s^k}{k!}\|w\gamma\|^k = M e^{s\|\gamma\||w|}.
\end{align*}
Thus, $f_\gamma:\zC\to\zC$ is a holomorphic function of exponential type without zeros. Then there exist
complex constants $C(\gamma)$ and $p(\gamma)$ such that
$f_\gamma(w)=C(\gamma)e^{p(\gamma)w}$.

Note that $C=C(\gamma)$ is independent of $\gamma$ because
$$
C(\gamma)=f_\gamma(0)=\beta(\varphi)(0)=T(1)(0).
$$

We also have that $f'_\gamma(0)=Cp(\gamma)=T(\gamma)(0)$. Thus we get that
$p(\gamma)=\frac{1}{C}T(\gamma)(0)$ is a linear continuous functional.
Finally, we get that $\beta(\varphi)(\gamma)=C e^{p(\gamma)}$ with $p\in E''$ and $C\neq 0$.
This implies that $Z(\beta(\varphi)-\lambda)$ is infinite for every $\lambda\neq 0$, as we wanted to prove.

To prove that $T$ satisfies the hypotheses of Proposition \ref{FHC modified} we follow the ideas of the second
proof of \cite[Theorem 1.3]{BonGro06}. Parameterizing the eigenvectors $e^\gamma$ it is possible to
construct a family of $C^2$-functions $C_k:\T\to \hbu(E)$ such that $T(C_k(\lambda))=\lambda C_k(\lambda)$ and
such that, for every Borel set of full Lebesgue measure, $B\subset\T$, the linear span of $\{C_k(\lambda):
\lambda\in B, \, k\in\N\}$ is dense in $\hbu(E)$. For $j\in\mathbb{Z}$ and $k\in\N$ set
$$
x_{k,j}=\int_\T \lambda^j C_k(\lambda) d\lambda,
$$
where the integral is in the sense of Riemann and $X_0=span\{x_{k,j}; j\in\mathbb{Z},\, k\in\N \}$. It follows from the proof of \cite[Th\'eor\`eme 2.2.]{BayGri04} that $X_0$ is dense in $\hbu(E)$ and that for $n\ge 0$, $j\in\mathbb{Z}$, $k\in\N$ we get
$$
T^nx_{k,j}=\int_\T \lambda^{j+n} C_k(\lambda) d\lambda.
$$
For every $y\in X_0$, there exists a linear combination $y=\sum_{l=1}^{m_y} a_l x_{k_l,j_l}$. So, we define
$$
u_n(y)=\sum_{l=1}^{m_y} a_l x_{k_l,j_l-n}.
$$
Finally, we have that $u_0(y)=y$ and that $T^iu_n(y)=u_{n-i}(y)$ if $i\leq n$, for every $y\in X_0$. Since
each $C_k$ is a $C^2$-function, by \cite[Lemma 9.23 (b)]{GroPer11}, we obtain that the series
$\sum_{n=1}^\infty T^n x_{k,j}$, $\sum_{n=1}^\infty u_n (x_{k,j})$ converge unconditionally for all
$j\in\mathbb{Z}$, $k\in\N$. As we claimed, $T$ satisfies the hypotheses of Proposition \ref{FHC modified}, and
so there exists a frequently hypercyclic subspace.
\end{proof}

\subsection{Translation Operators.} Suppose that $\tau_{z_0}:\hbu(E) \to \hbu(E)$ is the translation operator defined by $\tau_{z_0}(f)(z)=f(z+z_0)$. Next proposition is similar to \cite[Theorem 9.26]{GroPer11}, but in this case for translation operators in $\hbu(E)$, which gives sharp exponential growth conditions for frequently hypercyclic functions.

\begin{proposition}\label{crec traslacion}
Let $\u=\{\u_k\}_k$ be a holomorphy type and let $E$ be a Banach space with separable dual such that finite type polynomials are dense in $\u_k(E)$ for every $k$. Suppose that $\u$ is weakly differentiable with constants $c_{k,l}$ satisfying (\ref{constantes}). Let $\tau_{z_0}:\hbu(E) \to \hbu(E)$ be the translation operator by a non-zero vector $z_0\in E$. Then,
\begin{itemize}
\item[(a)] Given $\varepsilon>0$, then there exists $C>0$ and an entire function $f\in\hbu(E)$ which is frequently hypercyclic for $\tau_{z_0}$ and satisfies
$$
|f(z)|\leq Ce^{\varepsilon\|z\|}.
$$
\item[(b)] Let $\varepsilon:\R_+\rightarrow\R_+$ be a function such that $\liminf\limits_{r\rightarrow\infty}\varepsilon(r)=0$ and $C$ any positive number.
    Then there is no frequently hypercyclic entire function $f\in\hbu(E)$ for $\tau_{z_0}$, satisfying
$$
|f(z)|\leq Ce^{\varepsilon(\|z\|)\|z\|}, \text{ for all } z.
$$
\end{itemize}
\end{proposition}

\begin{proof}

(a) Note that $\tau_{z_0}(e^\gamma)=e^{\gamma(z_0)}e^\gamma$, thus
$$\inf\{\|\gamma\|,\,\gamma\in E'\text{ such that } |\tau_{z_0}(e^\gamma)(0)|=1\}=0.$$
It follows from Theorem \ref{ultimo} that for any $\varepsilon >0$, there exist a frequently hypercyclic function $f \in \hbu(E)$ such that
$$
|f(z)|\leq Ce^{\varepsilon\|z\|},
$$
for some positive constant $C$.

(b) Suppose that there exist a frequently hypercyclic function $f$ for $\tau_{z_0}$ such that $|f(z)|\leq Ce^{\varepsilon(\|z\|)\|z\|}$. Consider the complex line $L=\{\lambda z_0, \lambda\in\C\}$ and the restriction map given by
\begin{align*}
\hbu(E)&\longrightarrow \zH(\C)\\
g&\mapsto g|_{L}(\lambda)=g(\lambda z_0).
\end{align*}
Consider the following diagram
$$
\xymatrix{\hbu(E) \ar[r]^{\tau_{z_0}} \ar[d] & \hbu(E) \ar[d] \\ \zH(\C) \ar[r]_{\tau_1} & \zH(\C)  }
$$

Note that is a commutative diagram, for $g\in\hbu(E)$
$$
(\tau_{z_0}g)|_L(\lambda)=\tau_{z_0}g(\lambda z_0)=g((\lambda+1)z_0)=(g|_L)(\lambda+1)=\tau_1(g|_L)(\lambda).
$$
Also the restriction map has dense range: take $\gamma\in E'$ such that $\gamma(z_0)=1$,
then $\gamma^k|_L(\lambda)=\gamma^k(\lambda z_0)=\lambda^k$.  Thus, all polynomials belong to the range of the
restriction map.

Applying the hypercyclic comparison principle we get that $\tau_1$ is frequent hypercyclic and that $f|_L\in \zH(\C)$
is a frequently hypercyclic function that satisfies
$$
|f|_L(z)|=|f(\lambda z_0)|\leq Ce^{\varepsilon(\|\lambda z_0\|)\|\lambda z_0\|}.
$$
But this bound contradicts \cite[Theorem 9.26]{GroPer11}, which states that there is no such a function in $\zH(\C)$.
\end{proof}

\begin{remark}\rm
As we mentioned in the proof of the last proposition, in \cite{BlaBonGro10,GroPer11} it is proved that, given $\varepsilon$ such that $\liminf\limits_{\lambda\rightarrow\infty}\varepsilon(|\lambda|)=0$, there are not frequently hypercyclic functions for the translation operator in $\mathcal H(\zC)$ satisfying that $|f(\lambda)|\leq Ce^{\varepsilon(|\lambda|)|\lambda|}$. In contrast, there are hypercyclic functions of arbitrary slow growth (see \cite{Duy83}). The corresponding result in the Banach space setting has not been studied, up to our knowledge.
\end{remark}

\subsection{Differentiation Operators.} For the differentiation operator on $\hbu(E)$, $\mathcal{D}_a:\hbu(E)\to \hbu(E), \;\mathcal{D}_a(f)=d^1f(\cdot)(a),$ we can estimate the exponential type for the frequent hypercyclic functions. Since
$\mathcal{D}_a(e^\gamma)(0)=d^1(e^\gamma)(0)(a)=\gamma(a)$, we get that
$$
\inf\{\|\gamma\|,\,\gamma\in E' \text{ such that } |\mathcal{D}_a(e^\gamma)(0)|=1\}=\|a\|.
$$

Thus given $\varepsilon>0$ there exist a  frequently hypercyclic function $f$ such that
$$
|f(x)|\leq Ce^{\left(\|a\|+\varepsilon\right)\|x\|},
$$
for some $C>0$. It is not difficult to see that the best exponential type of a hypercyclic function for $\mathcal D_a$ is $\|a\|$. To prove this fact it suffices to conjugate $\mathcal D_a$ by the one dimensional differentiation operator (as we did in the proof of Proposition \ref{crec traslacion}) and apply \cite[Theorem 4.22]{GroPer11} (see also \cite{Gro90} and \cite{Shk93}).

\medskipç

{\bf Acknowledgements.} The authors would like to thanks to Professor A. Peris for his useful comments and to the anonymous referee for his/her carefully reading and suggestions which led to an improvement of the last section of this manuscript.

\bibliography{biblio}
\bibliographystyle{plain}

\end{document}